\pdfoutput=1
\documentclass[11pt]{article}
\usepackage{authblk}
\usepackage[toc,page]{appendix}

\usepackage{color}
\usepackage{helvet}         
\usepackage{courier}        
\usepackage{type1cm}        
\usepackage{framed}
\usepackage{tikz}
\usepackage{makeidx}         
\usepackage{graphicx}        
\usepackage{multicol}        
\usepackage[bottom]{footmisc}
\usepackage{amsmath}
\usepackage{amssymb}
\usepackage{bbold}
\usepackage{amsthm}
\usepackage{subcaption}
\usepackage{sidecap}
\usepackage{floatrow}
\usepackage{pdflscape}
\usepackage{bm}
\usepackage{comment}
\usepackage[font=small]{caption}
\usepackage[top=2cm, bottom=2cm, left=2cm, right=2cm]{geometry}
\newtheorem{theorem}{Theorem}
\newtheorem{corollary}[theorem]{Corollary}

\newtheorem{lemma}[theorem]{Lemma}

\numberwithin{theorem}{section}
\numberwithin{figure}{section}
\numberwithin{equation}{section}

\DeclareMathOperator{\SLE}{SLE}
\DeclareMathOperator{\hSLE}{hSLE}

\DeclareMathOperator{\diam}{diam}

\DeclareMathOperator{\Leb}{Leb}

\begin{document}

\title{Decomposition of Hypergeometric SLE and Reversibility}
\bigskip{}
\author{Mingchang Liu\thanks{liumc\_prob@163.com}}
\affil{KTH Royal Institute of Technology, Sweden}
\date{}

\maketitle

\begin{center}
\begin{minipage}{0.95\textwidth}
\abstract{ 
In this paper, we consider hypergeometric SLE process for $\kappa\in (4,8)$ and $\nu>\frac{\kappa}{2}-6$. Though the definition of hypergeometric SLE process is complicated, we show that given its hitting point on a specific boundary, its conditional law can be described by $\SLE_\kappa(\underline\rho)$ process. Based on this observation, by constructing a pair of curves, we derive the reversibility of hypergeometric SLE for $\kappa\in(4,8)$ and $\nu>-2$.
}

\bigskip{}


\end{minipage}
\end{center}


%

\tableofcontents

\newcommand{\eps}{\epsilon}
\newcommand{\ov}{\overline}
\newcommand{\U}{\mathbb{U}}
\newcommand{\T}{\mathbb{T}}
\newcommand{\HH}{\mathbb{H}}
\newcommand{\LA}{\mathcal{A}}
\newcommand{\LB}{\mathcal{B}}
\newcommand{\LC}{\mathcal{C}}
\newcommand{\LD}{\mathcal{D}}
\newcommand{\LF}{\mathcal{F}}
\newcommand{\LK}{\mathcal{K}}
\newcommand{\LE}{\mathcal{E}}
\newcommand{\LG}{\mathcal{G}}
\newcommand{\LL}{\mathcal{L}}
\newcommand{\LM}{\mathcal{M}}
\newcommand{\LQ}{\mathcal{Q}}
\newcommand{\LP}{\mathcal{P}}
\newcommand{\LR}{\mathcal{R}}
\newcommand{\LT}{\mathcal{T}}
\newcommand{\LS}{\mathcal{S}}
\newcommand{\LU}{\mathcal{U}}
\newcommand{\LV}{\mathcal{V}}
\newcommand{\LX}{\mathcal{X}}
\newcommand{\PartF}{\mathcal{Z}}
\newcommand{\LH}{\mathcal{H}}
\newcommand{\R}{\mathbb{R}}
\newcommand{\C}{\mathbb{C}}
\newcommand{\N}{\mathbb{N}}
\newcommand{\Z}{\mathbb{Z}}
\newcommand{\E}{\mathbb{E}}
\newcommand{\PP}{\mathbb{P}}
\newcommand{\QQ}{\mathbb{Q}}
\newcommand{\A}{\mathbb{A}}
\newcommand{\one}{\mathbb{1}}
\newcommand{\bn}{\mathbf{n}}
\newcommand{\MR}{MR}
\newcommand{\cond}{\,|\,}
\newcommand{\la}{\langle}
\newcommand{\ra}{\rangle}
\newcommand{\tree}{\Upsilon}
\newcommand{\prob}{\mathbb{P}}
\renewcommand{\Im}{\mathrm{Im}}
\renewcommand{\Re}{\mathrm{Re}}
\newcommand{\ii}{\mathfrak{i}}

\section{Introduction}
\label{sec::intro}
SLE process is introduced by O. Schramm in~\cite{SchrammScalinglimitsLERWUST}, which is defined in simply connected domains with two boundary points. This is a family of random curves with one parameter $\kappa>0$. $\SLE_\kappa$ curves are characterized by conformal invariance and domain Markov property. Thus, they are conjectured to be candidates of the scaling limits of interfaces of $2$D lattice models at criticality. The convergences of several models have been proved: loop-erased random walk and uniform spanning tree~\cite{LawlerSchrammWernerLERWUST}, percolation~\cite{SmirnovPercolationConformalInvariance}, level lines of Gaussian free field ~\cite{SchrammSheffieldDiscreteGFF}, Ising and FK-Ising model~\cite{ChelkakSmirnovIsing, CDCHKSConvergenceIsingSLE}.
In this paper, we will consider two natural variants of $\SLE$ process: hypergeometric $\SLE$ process and $\SLE_\kappa(\underline\rho)$ process. They are all conformally invariant random curves.

Hypergeometric $\SLE$ process is introduced by Dapeng Zhan in~\cite{ZhanReversibilityMore} (and he called it ``intermediate $\SLE$ process"), which is used to describe the law of the time-reversal of chordal $\SLE_\kappa(\rho)$ curves. Then, this definition is generalized in~\cite{QianConformalRestrictionTrichordal} and~\cite{WuHyperSLE} and the notations hypergeometric $\SLE$ and $\hSLE$ appeared.
\footnote{The notations hypergeometric $\SLE$ and $\hSLE$ are first introduced in~\cite{QianConformalRestrictionTrichordal}, where the author focuses on the case $\kappa=\frac{8}{3}$. When $\kappa\in(0,4)$ and $\nu>-4$, the hypergeometric SLE in~\cite{WuHyperSLE} is a subfamily of the hypergeometric SLE in~\cite{QianConformalRestrictionTrichordal}. See a more detailed comparison in~\cite[Appendix D]{HanLiuWuUST}.} In this paper, we will consider the case that $\kappa\in (4,8)$ and thus we will use the notations and definitions in~\cite{WuHyperSLE}. When $\kappa\in(0,8)$, hypergeometric $\SLE$ is a family of  random curves with two parameters $\kappa$ and $\nu$. It is defined in a quad $(\Omega;a,x,y,b)$ (where $\Omega$ is a simply connected domain, and $(a,x,y,b)$ are four marked boundary points located counterclockwisely), with starting point $a$, target point $b$ and two force points $x,y$. It can be used to describe the scaling limits of discrete models with alternating boundary conditions, for instance, see~\cite{KemppainenSmirnovFKIsingHyperSLE} for FK-Ising model,~\cite{WuHyperSLE} for Ising model and~\cite{HanLiuWuUST} for uniform spanning tree.  

$\SLE_\kappa(\underline\rho)$ process is a variant of $\SLE$ process which 
keeps tracks of several boundary points. $\SLE_\kappa(\underline\rho)$ process can be interpreted as flow lines of Gaussian free field, see~\cite{MillerSheffieldIG1},~\cite{MillerSheffieldIG2},~\cite{MillerSheffieldIG3} and~\cite{MillerSheffieldIG4}. In these works, the authors use the coupling between $\SLE_\kappa(\underline\rho)$ curves and Gaussian free field to obtain many properties of $\SLE_\kappa(\underline\rho)$ curves. (For example, the duality, reversibility, interactions among several flow lines and counterflow lines and so forth.)

In this paper, we first derive a connection between $\hSLE_\kappa(\nu)$ and $\SLE_\kappa(\underline\rho)$ when $\kappa\in (4,8)$ and $\nu>\frac{\kappa}{2}-6$. By~\cite[Lemma 3.4]{WuHyperSLE}, when $\kappa\in (4,8)$ and $\nu>\frac{\kappa}{2}-6$, the $\hSLE_\kappa(\nu)$ curve hits $(y,+\infty)$ almost surely.
\begin{theorem}\label{thm::decom}
Fix $\kappa\in (4,8)$ and $\nu>\frac{\kappa}{2}-6$. Let $\eta$ be the $\hSLE_\kappa(\nu)$ curve in $(\HH;0,x,y,\infty)$ from $0$ to $\infty$ with force points $x$, $y$. We denote by $\tau$ the hitting time of $\eta$ at $(y,+\infty)$. Then, the density of $\eta(\tau)$ is given by
\begin{equation}\label{eqn::dis}
\rho(u):=\frac{u^{-\frac{4}{\kappa}}(u-y)^{\frac{2\nu+12-2\kappa}{\kappa}}\left(u-x\right)^{-\frac{2\nu+4}{\kappa}}}{\int_y^{+\infty} s^{-\frac{4}{\kappa}}(s-y)^{\frac{2\nu+12-2\kappa}{\kappa}}\left(s-x\right)^{-\frac{2\nu+4}{\kappa}}ds},\quad\text{for }u\in(y,+\infty).
\end{equation}
Moreover, the conditional distribution of $\eta[0,\tau]$ given $\eta(\tau)$ is $\SLE_\kappa(\nu+2,\kappa-6-\nu,-4)$ in $\HH$ from $0$ to $\infty$ with the marked points $x$, $y$ and $\eta(\tau)$.
\end{theorem}
Next, we will consider the time-reversal of $\hSLE_\kappa(\nu)$. The reversibility is natural from the perspective that $\SLE$ process and its variants describe the scaling limit of discrete models. It is not clear how to derive the reversibility directly just from the driving function. In~\cite{ZhanReversibility},~\cite{MillerSheffieldIG2} and~\cite{MillerSheffieldIG3}, the authors derive the reversibility of $\SLE_\kappa$ and of $\SLE_\kappa(\underline\rho)$ in several cases. In~\cite{yu2023timereversal}, the author derives the law of the time-reversal of $\SLE_\kappa(\underline\rho)$ in general case.
In this paper, we will prove the reversibility of $\hSLE_\kappa(\nu)$ for $\kappa\in(4,8)$ and $\nu>-2$. 
\begin{theorem}\label{thm::reversibility} 
Define the anti-conformal map $\psi$ from $\HH$ onto itself by $\psi(z):=\frac{1}{\overline z}$ for $z\in\HH$. Let $\eta$ be the $\hSLE_\kappa(\nu)$ curve in $(\HH;0,x,y,\infty)$ from $0$ to $\infty$ with force points $x$, $y$. Then, $\psi(\eta)$ has the same law as the $\hSLE_\kappa(\nu)$ curve in $(\HH;0,\psi(y),\psi(x),\infty)$ from $0$ to $\infty$ with force points $\psi(y)$ and $\psi(x)$.
\end{theorem}
Note that the case $\nu=-2$ is special: in this case, the law of $\hSLE_\kappa(\nu)$ equals the standard $\SLE_\kappa$ (see Section~\ref{sec::hyper}). The reversibility of  the $\SLE_\kappa$ curve for $\kappa\in (4,8)$ has been proved in~\cite[Theorem 1.2]{MillerSheffieldIG3}. In~\cite{WuHyperSLE}, the author proves Theorem~\ref{thm::reversibility} for $\nu\ge\frac{\kappa}{2}-4$ and $\kappa\in (0,+\infty)$. In that case, almost surely, $\hSLE_\kappa(\nu)$ curve does not hit $(x,y)$, which is essential for the proof there. In this paper, we will use a new method to prove the reversibility of the $\hSLE_\kappa(\nu)$ curve for $\kappa\in (4,8)$ and $\nu>-2$. 

The paper is organized as follows. 
\begin{itemize}
\item
In Section~\ref{sec::preliminaries}, we recall the definition of $\hSLE_\kappa(\nu)$ and $\SLE_\kappa(\underline\rho)$. 
\item
In Section~\ref{sec::connection}, we complete the proof of Theorem~\ref{thm::decom} by constructing suitable martingales.
\item
In Section~\ref{sec::proof}, we complete the proof of Theorem~\ref{thm::reversibility} by constructing a pair of random curves $(\eta_1,\eta_2)$. The marginal law of  $\eta_1$ equals $\hSLE_\kappa(\nu)$ and the conditional law of $\eta_1$ given $\eta_2$ equals $\SLE_\kappa$. We conclude the proof by showing the reversibility of $\eta_2$. The proof is based on the reversibility of $\SLE_\kappa(\rho)$~\cite[Theorem 1.2]{MillerSheffieldIG3}.
\end{itemize}
In this paper, we define $B(z,r):=\{w\in\C: |w-z|<r\}$ for $z\in\C$.


\section{Preliminaries}
\label{sec::preliminaries}

\subsection{Curve space}
\label{sec::curve}
A curve is defined by a continuous map from $[0, 1]$ to $\C$. Let $\mathcal{C}$ be the space of unparameterized curves in $\C$.  Define the metric on $\mathcal{C}$ as follows:
\begin{align}\label{eqn::curves_metric}
d(\gamma_1, \gamma_2):=\inf\sup_{t\in[0,1]}\left|\hat{\gamma}_1(t)-\hat{\gamma}_2(t)\right|,
\end{align}
where the infimum is taken over all the choices of  parameterizations  $\hat{\gamma}_1$ and $\hat{\gamma}_2$ of $\gamma_1$ and $\gamma_2$. The metric space $(\mathcal{C}, d)$ is complete and separable.

\subsection{Loewner chain and $\SLE$}
\label{sec::SLE}
An $\HH$-hull $K$ is a compact set of $\overline{\HH}$ such that $\HH\setminus K$ is simply connected. For a collection of growing $\HH$-hulls $(K_t:t\ge 0)$, they corresponds to a families of conformal maps $(g_t: t\ge 0)$ such that for every $t\ge 0$, $g_t$ is the conformal map from the unbounded connected component of $\HH\setminus K_t$ onto $\HH$ with the following normalization:
\[\lim_{z\to\infty}|g_t(z)-z|=0.\]
The conformal maps $(g_t: t\ge 0)$ satisfies the following Loewner equation: for $z\in\overline{\HH}$, 
\[\partial_t g_t(z)=\frac{2}{g_t(z)-W_t},\quad g_0(z)=z, \]
where $(W_t: t\ge 0)$ is a real-valued function, which is called the driving function. When $W_t=\sqrt\kappa B_t$, the corresponding $\HH$-hulls $(K_t:t\ge 0)$ is the $\SLE_\kappa$ process. In~\cite{RohdeSchrammSLEBasicProperty}, it has been proved that $\SLE_\kappa$ process is generated by a continuous curve almost surely for $\kappa\neq 8$. The same result for the case $\kappa=8$ is obtained from the convergence of uniform spanning tree~\cite{LawlerSchrammWernerLERWUST}. 

\subsection{Hypergeometric $\SLE$}
\label{sec::hyper}
The definition of $\hSLE_\kappa(\nu)$ involves the hypergeometric function. We define the function $F:[0,1]\to \R$ as
\[F(z):={}_2F_1\left(\frac{2\nu+4}{\kappa},1-\frac{4}{\kappa},\frac{2\nu+8}{\kappa};z\right).\]
The function $F$ satisfies the following ODE:
\begin{equation}\label{eqn::ODE}
z(1-z)F''(z)+\left(\frac{2\nu+8}{\kappa}-\frac{2\nu+2\kappa}{\kappa}z\right)F'(z)-\frac{2\nu+4}{\kappa}\times\left(1-\frac{4}{\kappa}\right) F(z)=0.
\end{equation}
When $\kappa\in(4,8)$ and $\nu>\frac{\kappa}{2}-6$, $F$ is a bounded and positive function on $[0,1]$ and we have the following integral formula (for instance, see~\cite[Eq.~15.3.1]{AbramowitzHandbook}): 
\begin{equation}\label{eqn::integral}
F(z)=\frac{1}{B\left(\frac{\kappa-4}{\kappa},\frac{2\nu+12-\kappa}{\kappa}\right)}\int_0^1 x^{-\frac{4}{\kappa}}(1-x)^{\frac{2\nu+12-2\kappa}{\kappa}}(1-zx)^{-\frac{2\nu+4}{\kappa}} dx,
\end{equation}
where $B(\cdot,\cdot)$ is the Beta function.
Note that the requirement $\kappa\in(4,8)$ and $\nu>\frac{\kappa}{2}-6$ implies that the integral is finite. The $\hSLE_\kappa(\nu)$ process from $0$ to $\infty$ with force points $x$ and $y$ is a random Loewner chain, whose driving function satisfies the following SDEs:
\begin{align}\label{eqn::hyperSDE}
\begin{cases}
d W_t=\sqrt{\kappa}d B_t + \frac{(\nu+2)dt}{W_t-V_t^{x}}+\frac{-(\nu+2)dt}{W_t-V_t^{y}}-\kappa \frac{F'\left(\frac{V_t^x-W_t}{V_t^y-W_t}\right)}{F\left(\frac{V_t^x-W_t}{V_t^y-W_t}\right)}\times\frac{V_t^y-V_t^x}{(V_t^y-W_t)^2}d t, \quad W_0=0 ; \\
d V_t^x = \frac{2 \, d t}{V_t^x-W_t}, \quad d V_t^y = \frac{2 \, d t}{V_t^y-W_t},\quad V_0^x=x \text{ and }V_0^y=y;
\end{cases}
\end{align}
where $(B_t)_{t\ge 0}$ is the standard one-dimentional Brownian motion. In particular, if $\nu=-2$, the $\hSLE_\kappa(\nu)$ curve equals the standard $\SLE_\kappa$ curve. In~\cite[Proposition 3.3]{WuHyperSLE}, the author proves that~\eqref{eqn::hyperSDE} is well-defined for all time. Moreover, when $\kappa\in(4,8)$ and $\nu>\frac{\kappa}{2}-6$, almost surely, $\hSLE_\kappa(\nu)$ process is generated by a continuous curve from $0$ to $\infty$ and it hits $(y,\infty)$. After the hitting time at $(y,\infty)$, $\hSLE_\kappa(\nu)$ curve evolves as the standard $\SLE_\kappa$ process.

\subsection{$\SLE_\kappa(\underline\rho)$ process and Gaussian free field}
\label{sec::reversibility}
We denote the $\SLE_\kappa(\rho_1^L,\ldots,\rho_n^L;\rho_1^R,\ldots,\rho_m^R)$ process by  $\SLE_\kappa(\underline\rho)$ for simplicity. The $\SLE_\kappa(\underline\rho)$ process starting from $0$ to $\infty$ with marked points $x_n^L<\cdots<x_1^L\le 0\le x_1^R<\cdots<x_m^R$ and weights $(\rho_1^L,\ldots,\rho_n^L;\rho_1^R,\ldots,\rho_m^R)$ is a random Loewner chain whose driving function $W$ satisfies the following SDEs:
\begin{align}\label{eqn::rhoSDE}
\begin{cases}
d W_t=\sqrt{\kappa}d B_t + \sum_{1\le i\le n}\frac{\rho_i^L dt}{W_t-V_t^{i,L}}+\sum_{1\le i\le m}\frac{\rho_i^R dt}{W_t-V_t^{i,R}},\quad W_0=0; \\
d V_t^{i,L} = \frac{2d t}{V_t^{i,L}-W_t},\quad V_0^{i,L}=x_i^L \text{ for }1\le i\le n;\quad d V_t^{i,R} = \frac{2d t}{V_t^{i,R}-W_t},\quad V_0^{i,R}=x_i^R \text{ for }1\le i\le m;
\end{cases}
\end{align}
where $(B_t)_{t\ge 0}$ is the standard one-dimentional Brownian motion. 
We emphasize that $x_1^L=0$ and $x_1^R=0$ are allowed. In this case, we always have $V_t^{1,L}\le 0$ and $V_t^{1,R}\ge 0$ and we will denote by $x_1^L=0^-$ and denote by $x_1^R=0^+$. 
It has been shown in~\cite{MillerSheffieldIG1} that~\eqref{eqn::rhoSDE} has a unique solution up to the continuation threshold: the first time $t$ such that $W_t=V_t^{j,q}$ and $W_t\neq V_t^{j+1,q}$ and $\sum_{i=1}^j\rho_i^q\le -2$ for some $j$ and $q\in\{L,R\}$. (Here we use the convention that $V_t^{n+1,L}=-\infty$ and $V_t^{m+1,R}=+\infty$.) 

Now, we introduce the definition of GFF (see~\cite{SheffieldGFFMath} for more details). Suppose $\Omega$ is a simply connected domain. The GFF on $\Omega$ with boundary value $h_\partial$, which we denote by $h+h_\partial$, is a random distribution on $\Omega$ such that for every $f\in C^\infty_c(\Omega)$, $(h,f)$ is a Gaussian variable with
\[\E[(h,f)]=\int \hat h_\partial(x)f(x)d x,\quad\text{and}\quad \E[(h,f)^2]=\int_\Omega f(x)G_\Omega(x,y)f(y)d xd y,\]
where $\hat h_\partial$ is the harmonic extension of $h_\partial$ onto $\Omega$ and $G_\Omega(\cdot,\cdot)$ is the Green's function with Dirichlet boundary condition and it satisfies that $G_\Omega(x,y)+\log |x-y|$ is bounded when $y$ tends to $x$.

Fix $x_n^L<\cdots<x_1^L\le 0\le x_1^R<\cdots<x_m^R$ and $(\rho_1^L,\ldots,\rho_n^L;\rho_1^R,\ldots,\rho_m^R)$. Fix $\kappa>0$ and define $\lambda:=\frac{\pi}{\sqrt\kappa}$.
We define $u$ to be the bounded harmonic function with the following boundary conditions: $u$ equals $\lambda(1+\sum_{j=0}^i \rho_j^R)$ on the interval $(x_i^R,x_{i+1}^R)$ for $0\le i\le m$ and  $u$ equals $-\lambda(1+\sum_{j=0}^i \rho_j^L)$ on the interval $(x_{i+1}^L,x_{i}^L)$ for $0\le i\le n$, where we use the convention that $x_0^L=0^-$ and $x_0^R=0^+$ and $\rho_0^L=\rho_0^R=0$ and $x_{n+1}^{L}=-\infty$ and $x_{m+1}^R=+\infty$.
In~\cite{MillerSheffieldIG1}, the authors establish the coupling of GFF and the $\SLE_\kappa(\rho_1^L,\ldots,\rho_n^L;\rho_1^R,\ldots,\rho_m^R)$ curve, such that under this coupling, the $\SLE_\kappa(\rho_1^L,\ldots,\rho_n^L;\rho_1^R,\ldots,\rho_m^R)$ curve can be viewed as flow line of $h+u$. Several flow lines can be coupled with the same GFF. In this paper, we will use the following property: if $y_1<y_2$ and $\eta_1$ is the flow line of $h-\theta_1$ and $\eta_2$ is the flow line of $h-\theta_2$ such that $\theta_1<\theta_2$, then almost surely, $\eta_2$ stays to the right of $\eta_1$. See~\cite{MillerSheffieldIG1} for more details.

\label{reversibilityhsle}

\section{Proof of Theorem~\ref{thm::decom}}
\label{sec::connection}
In this section, we will prove Theorem~\ref{thm::decom}. Suppose $\kappa\in (4,8)$ and $\nu>\frac{\kappa}{2}-6$ and fix $0<x<y$. Suppose $\eta$ is the $\hSLE_{\kappa}(\nu)$ curve on $\HH$ with force points $x$, $y$. We denote by $W$ the driving function of $\eta$ and by $(g_t:t\ge 0)$ the corresponding conformal maps. For $z\in \R$, we denote by $(V_t(z): t\ge 0)$ the corresponding Loewner evolution: $(V_t(z): t\ge 0)$ is a bounded variational process which satisfies the following equation
\begin{equation}\label{eqn::loewner}
dV_t(z)=\frac{2dt}{V_t(z)-W_t},\quad V_0(z)=z,\quad \text{for }t\ge 0.
\end{equation}
By~\cite[Lemma 3.3]{MillerSheffieldIG2}, if $\Leb(\eta\cap\R)=0$, then $V_t(z)$ is well-defined for $t\in(0,+\infty)$. Moreover, for every $z\ge 0$, $V_t(z)$ is the image of $\max\{z,\text{ the rightmost point of }\eta[0,t]\cap\R\}$ under $g_t$; and for every $z\le 0$, $V_t(z)$ is the image of $\min\{z,\text{ the leftmost point of }\eta[0,t]\cap\R\}$ under $g_t$. 
\begin{lemma}\label{lem::variationalprocess}
Almost surely, we have $\Leb(\eta\cap\R)=0$.
\end{lemma}
\begin{proof}
Recall that we denote by $\tau$ the hitting time of $\eta$ at $(y,+\infty)$. After time $\tau$, the random curve $\eta$ will evolve as a standard $\SLE_{\kappa}$ in the remaining domain. Thus, we have almost surely,
\begin{equation}\label{eqn::aux000}
\Leb(\eta[\tau,+\infty)\cap\R)=0.
\end{equation}
Before time $\tau$, the random curve $\eta$ is absolutely continuous with respect to $\SLE_{\kappa}(\nu+2,\kappa-6-\nu)$ curve with marked points $x$ and $y$ (see discussion after~\cite[Proposition 3.3]{WuHyperSLE}).
Thus, before time $\tau$, the random curve $\eta$ is absolutely continuous with respect to $\gamma\sim\SLE_{\kappa}(\nu+2)$ with marked point $x$.
By~\cite[Lemma 2.5]{WuHyperSLE}, for every $w\in (0,+\infty)$, we have
\[\PP[w\in\gamma]=0.\]
This implies that almost surely, $\Leb(\gamma\cap\R)=0$. 
From the absolute continuity, we have that almost surely,
\[\Leb(\eta[0,\tau]\cap\R)=0.\]
Combining with~\eqref{eqn::aux000}, we complete the proof.
\end{proof}
\begin{lemma}\label{lem::deco}
\begin{itemize}
\item
Recall that we denote by $\tau$ the hitting time of $\eta$  at $(y,+\infty)$. For $z\in (y,+\infty)$, we define
\[M_t(z):=\frac{V'_t(z)(V_t(z)-W_t)^{-\frac{4}{\kappa}}(V_t(z)-V_t(x))^{-\frac{2(\nu+2)}{\kappa}}(V_t(z)-V_t(y))^{-\frac{2({\kappa}-6-\nu)}{\kappa}}(V_t(y)-W_t)^{\frac{\kappa-4}{\kappa}}}{F\left(\frac{V_t(x)-W_t}{V_t(y)-W_t}\right)}.\]
Then, $\{M_{t\wedge\tau}\}_{t\ge 0}$ is a local martingale for $\eta$. 
\item
Moreover, we define $\tau_\eps$ to be the hitting time of $\eta$ at the union of $\eps$-neighborhood of $(y,+\infty)$ and $\partial B\left(0,\frac{1}{\eps}\right)$. Then, $\{M_{t\wedge\tau_\eps}\}_{t\ge 0}$ is a martingale. Weighted by $\frac{M_{t\wedge\tau_\eps}}{M_0}$, the random curve $\eta[0,\tau_\eps]$ has the same law as the $\SLE_{\kappa}(\nu+2,{\kappa}-6-\nu,-4)$ stopped at the hitting time of the $\eps$-neighbour of $(y,+\infty)$.
\end{itemize}
\end{lemma}
\begin{proof}
For the first conclusion, we define $L_t$ to be the numerator of $M_t$ and define $F_t$ to be the denominator of $M_t$. Define $Z_t:=\frac{V_t(x)-W_t}{V_t(y)-W_t}$. By I\^to's formula, we have
\begin{align}\label{eqn::aux1}
dL_t=&L_t\left(\frac{4}{{\kappa}}\frac{1}{V_t(z)-W_t}+\frac{4-{\kappa}}{{\kappa}}\frac{1}{V_t(y)-W_t}\right)dW_t\notag\\
&+L_t\left(\frac{2(\nu+2)}{{\kappa}}\frac{2}{(V_t(z)-W_t)(V_t(x)-W_t)}-\frac{2(2+\nu)}{{\kappa}}\frac{2}{(V_t(z)-W_t)(V_t(y)-W_t)}\right)dt,
\end{align}
and
\begin{align}\label{eqn::aux2}
dF_t&=F'\left(Z_t\right)d Z_t+\frac{{\kappa}}{2}F''\left(Z_t\right)\times \frac{(V_t(x)-V_t(y))^2}{(V_t(y)-W_t)^4}dt\notag\\
&=F'\left(Z_t\right)\times\frac{V_t(y)-V_t(x)}{(V_t(y)-W_t)^2}\left(\frac{2dt}{V_t(x)-W_t}+\frac{(2-{\kappa})dt}{V_t(y)-W_t}-dW_t\right)+\frac{{\kappa}}{2}F''\left(Z_t\right)\times \frac{(V_t(x)-V_t(y))^2}{(V_t(y)-W_t)^4}dt.
\end{align}
and
\begin{align}\label{eqn::aux3}
dM_t(z)=\frac{dL_t}{F_t}-\frac{L_t}{F_t^2}d F_t+\frac{L_t}{F_t^3}d\langle F_t\rangle-\frac{1}{F_t^2}d\langle L_t,F_t\rangle.
\end{align}
Plugging~\eqref{eqn::hyperSDE},~\eqref{eqn::aux1},~\eqref{eqn::aux2}  into~\eqref{eqn::aux3}, we have
\begin{align*}
dM_t(z)
=M_t(z)\left(\frac{4}{
\sqrt{\kappa}}\frac{1}{V_t(z)-W_t}+\frac{4-{\kappa}}{\sqrt{\kappa}}\frac{1}{V_t(y)-W_t}+\sqrt{\kappa}\frac{F'(Z_t)}{F(Z_t)}\frac{V_t(y)-V_t(x)}{(V_t(y)-W_t)^2}\right)dB_t+dR_t,
\end{align*}
where $R_t$ is the drift term. Moreover, we have
\begin{align*}
dR_t=&-\frac{{\kappa} L_t}{2F_t^2}\times\frac{V_t(y)-V_t(x)}{(V_t(y)-W_t)^2(V_t(x)-W_t)}\\
&\times\left(Z_t(1-Z_t)F''(Z_t)+\left(\frac{2\nu+8}{{\kappa}}-\frac{2\nu+2{\kappa}}{{\kappa}}Z_t\right)F'(Z_t)-\frac{2\nu+4}{{\kappa}}\times\left(1-\frac{4}{{\kappa}}\right) F(Z_t)\right)dt.
\end{align*}
Thus, by~\eqref{eqn::ODE}, we have that the drift term of $\{M_{t\wedge\tau}\}_{t\ge 0}$ equals $0$ and this implies that $\{M_{t\wedge\tau}\}_{t\ge 0}$ is a local martingale for $\eta$.

For the second conclusion, note that for $t\le \tau_\eps$, $M_t(z)$ is uniformly bounded. Thus, $\{M_{t\wedge\tau_\eps}(z)\}_{t\ge 0}$ is a martingale. Recall the definition of $\eta$ given in~\eqref{eqn::hyperSDE}. By Girsanov theorem, we have that weighted by $M_{t\wedge\tau_\eps}/M_0$, the random curve $\eta[0,\tau_\eps]$ has the same law of the $\SLE_{\kappa}(\nu+2,{\kappa}-6-\nu,-4)$ stopped at the hitting time of the union of $\eps$-neighbour of $(y,+\infty)$ and $\partial B(0,\frac{1}{\eps})$. This completes the proof.
\end{proof}
\begin{proof}[Proof of Theorem~\ref{thm::decom}]
We first prove~\eqref{eqn::dis}. We will use the same notations as in Lemma~\ref{lem::deco}. Fix $w\in [y,+\infty)$. Define
\begin{align*}
N_t=N_t(w):=\frac{1}{B\left(\frac{\kappa-4}{{\kappa}},\frac{2\nu+12-{\kappa}}{{\kappa}}\right)}\int_{w}^{+\infty}M_{t}(z)dz.
\end{align*}
Recall that we denote by $\tau$ the hitting time of $\eta$ at $(y,+\infty)$.
By Lemma~\ref{lem::deco}, $\{N_{t\wedge\tau}\}_{t\ge 0}$ is a local martingale for $\eta$. Moreover, for $t<\tau$, we have
\begin{align}
N_t&=\frac{1}{B\left(\frac{\kappa-4}{{\kappa}},\frac{2\nu+12-{\kappa}}{{\kappa}}\right)}\int_{w}^{+\infty}M_{t}(z)dz\notag\\
&=\frac{\int_w^{+\infty}dz V'_t(z)(V_t(z)-W_t)^{-\frac{4}{\kappa}}(V_t(z)-V_t(x))^{-\frac{2(\nu+2)}{\kappa}}(V_t(z)-V_t(y))^{-\frac{2({\kappa}-6-\nu)}{\kappa}}(V_t(y)-W_t)^{\frac{\kappa-4}{\kappa}}}{B\left(\frac{\kappa-4}{{\kappa}},\frac{2\nu+12-{\kappa}}{{\kappa}}\right) F\left(\frac{V_t(x)-W_t}{V_t(y)-W_t}\right)}\notag\\
&=\frac{\int_{\frac{V_t(w)-W_t}{V_t(y)-W_t}}^{+\infty}u^{-\frac{4}{{\kappa}}}(u-1)^{\frac{2({\kappa}-6-\nu)}{{\kappa}}}\left(u-Z_t\right)^{\frac{2(\nu+2)}{{\kappa}}}du}{\int_0^1 x^{-\frac{4}{{\kappa}}}(1-x)^{\frac{2\nu+12-2{\kappa}}{{\kappa}}}(1-Z_tx)^{-\frac{2\nu+4}{{\kappa}}} dx}\tag{due to~\eqref{eqn::integral} and setting $u=\frac{V_t(z)-W_t}{V_t(y)-W_t}$}\\
&=\frac{\int_{\frac{V_t(w)-W_t}{V_t(y)-W_t}}^{+\infty}u^{-\frac{4}{{\kappa}}}(u-1)^{\frac{2({\kappa}-6-\nu)}{{\kappa}}}\left(u-Z_t\right)^{\frac{2(\nu+2)}{{\kappa}}}du}{\int_{1}^{+\infty}v^{-\frac{4}{{\kappa}}}(v-1)^{\frac{2({\kappa}-6-\nu)}{{\kappa}}}\left(v-Z_t\right)^{\frac{2(\nu+2)}{{\kappa}}}dv}.\tag{setting $v=\frac{1}{x}$}
\end{align}
This implies that $N_t\le 1$ for every $t<\tau$. Thus, $\{N_{t\wedge\tau}\}$ is a uniformly integrable martingale. \begin{itemize}
\item
On the event $\{\eta\text{ hits }(y,+\infty)\text{ at }(w,+\infty)\}$, one has
\begin{equation}\label{eqn::aux11}
\lim_{t\to \tau}\frac{V_{t}(w)-W_t}{V_t(y)-W_t}=1,\quad\lim_{t\to\tau} Z_t=1.
\end{equation}
The proof of~\eqref{eqn::aux11} is standard, for instance, see~\cite[Lemma B.2]{PeltolaWuGlobalMultipleSLEs}.
Thus, in this case, by dominated convergence theorem, we have
\[\lim_{t\to\tau} N_t=1.\]
\item
On the event $\{\eta\text{ hits }(y,+\infty)\text{ at }(y,w)\}$, one has
\begin{equation}\label{eqn::aux12}
\lim_{t\to\tau}\frac{V_{t}(w)-W_t}{V_t(y)-W_t}=+\infty.
\end{equation}
Thus, in this case, we have 
\[\lim_{t\to\tau}N_t=0.\]
\end{itemize}
Combining these two observations together, we have
\begin{equation}\label{eqn::aux13}
N_\tau:=\lim_{t\to\tau}N_t=\one_{\{\eta\text{ hits }(y,+\infty)\text{ at }(w,+\infty)\}}.
\end{equation}
Thus, since $\{N_{t\wedge\tau}\}_{t\ge 0}$ is uniformly integrable, the stopping time theorem implies that 
\begin{align*}\PP[\eta\text{ hits }(y,+\infty)\text{ at }(w,+\infty)]=\E[N_\tau]=\E[N_0]&=\frac{\int_{\frac{w}{y}}^{+\infty}u^{-\frac{4}{{\kappa}}}(u-1)^{\frac{2({\kappa}-6-\nu)}{{\kappa}}}\left(u-\frac{x}{y}\right)^{\frac{2(\nu+2)}{{\kappa}}}du}{\int_{1}^{+\infty}u^{-\frac{4}{{\kappa}}}(u-1)^{\frac{2({\kappa}-6-\nu)}{{\kappa}}}\left(u-\frac{x}{y}\right)^{\frac{2(\nu+2)}{{\kappa}}}du}\\
&=\int_{w}^{+\infty}\rho(u)du.
\end{align*}
This finishes the proof of~\eqref{eqn::dis}.

Next, we derive the conditional law of $\eta$ given $\eta(\tau)$. For $z\in(y,+\infty)$, we denote by $\eta_z$ the $\SLE_{\kappa}(\nu+2,{\kappa}-6-\nu,-4)$ curve from $0$ to $\infty$ with marked points $x$, $y$ and $z$. 
We define $\tau_{\eps,z}$ to be the hitting time of $\eta_z$ at the union of $\eps$-neighborhood of $(y,+\infty)$ and $\partial B\left(0,\frac{1}{\eps}\right)$. Define $\tau_z$ to be the hitting time of $\eta_z$ at $z$. Note that by~\cite[Theorem 1.3, Remark 5.3]{MillerSheffieldIG1}, almost surely, $\eta_z$ can not hit $(y,z)\cup(z,+\infty)$. Thus, $\eta_z$ either hits $z$ ($\tau_z<\infty$) or goes to $\infty$ ($\tau_z=+\infty$). We define
\[A:=\{z\in (y,+\infty): \PP[\tau_z=+\infty]>0\}.\]

Recall that by~\cite[Lemma 3.4]{WuHyperSLE}, $\tau<+\infty$ almost surely. Fix $R>0$.
\begin{align}
\E[\one_{\diam(\eta[0,\tau]))> R}]&\ge\lim_{\eps\to 0}\E[\one_{\diam(\eta[0,\tau_\eps]))> R}]\tag{due to the continuity of $\eta$}\\
&=\frac{1}{B\left(\frac{\kappa-4}{{\kappa}},\frac{2\nu+12-{\kappa}}{{\kappa}}\right)}{\lim_{\eps\to 0}}\int_{y}^{+\infty}\E\left[\one_{\diam(\eta[0,\tau_\eps]))> R}\frac{M_{\tau_\eps}(z)}{M_0(z)}\right]M_{0}(z)dz\notag\\
&=\frac{1}{B\left(\frac{\kappa-4}{{\kappa}},\frac{2\nu+12-{\kappa}}{{\kappa}}\right)}{\lim_{\eps\to 0}}\int_{y}^{+\infty}\E\left[\one_{\diam(\eta_z[0,\tau_{\eps,z}])>R}\right]M_{0}(z)dz\quad\tag{due to Lemma~\ref{lem::deco}}\\
&\ge\frac{1}{B\left(\frac{\kappa-4}{{\kappa}},\frac{2\nu+12-{\kappa}}{{\kappa}}\right)}\int_{y}^{+\infty}\E\left[\one_{\tau_z=+\infty}\right]M_{0}(z)dz\notag.
\end{align}
By letting $R\to\infty$, we have
\begin{equation}\label{eqn::aux00}
\Leb(A)=0.
\end{equation}
Suppose $f$ is a bounded continuous function on the curve space. Fix $w\in (y,+\infty)$ and define $g:=\one_{(w,+\infty)}$. We have
\begin{align}\label{eqn::conditionallaw}
\E[f(\eta[0,\tau])g(\eta(\tau))]&=\lim_{\eps\to 0}E[f(\eta[0,\tau_\eps])g(\eta(\tau))]\tag{due to the continuity of $\eta$}\\
&=\frac{1}{B\left(\frac{\kappa-4}{\kappa},\frac{2\nu+12-\kappa}{\kappa}\right)}\lim_{\eps\to 0}\E[\one_{\diam(\eta[0,\tau_\eps]))> R}N_{\tau_\eps}(w)]\notag\\
&=\frac{1}{B\left(\frac{\kappa-4}{{\kappa}},\frac{2\nu+12-{\kappa}}{{\kappa}}\right)}{\lim_{\eps\to 0}}\int_{y}^{+\infty}\E\left[f(\eta[0,\tau_\eps])\frac{M_{\tau_\eps}(z)}{M_0(z)}\right]g(z)M_{0}(z)dz\tag{due to~\eqref{eqn::dis}}\\
&=\frac{1}{B\left(\frac{\kappa-4}{{\kappa}},\frac{2\nu+12-{\kappa}}{{\kappa}}\right)}{\lim_{\eps\to 0}}\int_{y}^{+\infty}\E\left[f(\eta_z[0,\tau_{\eps,z}])\right]g(z)M_{0}(z)dz\quad\tag{due to Lemma~\ref{lem::deco}}\\
&=\frac{1}{B\left(\frac{\kappa-4}{{\kappa}},\frac{2\nu+12-{\kappa}}{{\kappa}}\right)}\int_{y}^{+\infty}\E\left[f(\eta_z[0,\tau_z])\right]g(z)M_{0}(z)dz. \quad\tag{due to the continuity of $\eta_z$ and~\eqref{eqn::aux00}}
\end{align}
Since the curve space is Polish, standard approximation method implies that the above result holds when both $f$ and $g$ are bounded and measurable. This proves that given $\eta(\tau)$, the conditional law of $\eta$ equals $\eta_z$.
This finishes the proof.
\end{proof}
From the explicit form in~\eqref{eqn::dis}, we can connect it to the conformal map $f$ in Figure~\ref{fig::conformal}.
\begin{figure}[ht!]
\begin{center}
\includegraphics[width=0.5\textwidth]{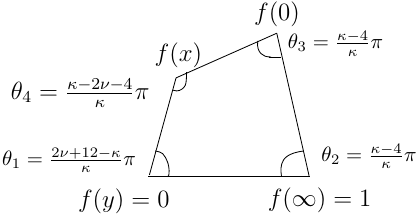}
\end{center}
\caption{\label{fig::conformal}
This is an illustration image under $f$. We require $f(y)=0$ and $f(\infty)$=1. The corresponding angles of the quad are shown in picture.}
\end{figure}
By Schwartz-Christoffel formula, the conformal map $f$ has the following explicit form:
\[f(z)=\frac{\int_{y}^z v^{-\frac{4}{\kappa}}(v-y)^{\frac{2\nu+12-2\kappa}{\kappa}}\left(v-x\right)^{-\frac{2\nu+4}{\kappa}}dv}{\int_{y}^\infty v^{-\frac{4}{\kappa}}(v-y)^{\frac{2\nu+12-2\kappa}{\kappa}}\left(v-x\right)^{-\frac{2\nu+4}{\kappa}}dv}.\]
\begin{corollary}
Assume the same setup as in Theorem~\ref{thm::decom}. Then, $f(\eta(\tau))$ has the uniform distribution on $[0,1]$.
\end{corollary}\label{coro::dis}
\begin{proof}
Note that the derivative of $f$ equals to $\rho$ (see~\eqref{eqn::dis}) on $(y,+\infty)$. This finishes the proof.
\end{proof}

\section{Proof of Theorem~\ref{thm::reversibility}}
\label{sec::proof}
In this section, we will prove Theorem~\ref{thm::reversibility}. 
Define $\psi$ to be
\[\psi(z):=\frac{1}{\overline z},\quad\text{for }z\in\HH.\]
Fix $\kappa\in (4,8)$ and $\nu>-2$. Fix $0<x<y$. 
Denote by $\eta_{\hSLE}$ the $\hSLE_\kappa(\nu)$ curve from $0$ to $\infty$ with force points $x$ and $y$. We will prove that the law of $\psi({\eta}_{\hSLE})$ equals the law of the $\hSLE_\kappa(\nu)$ curve from $0$ to $\infty$ with force points $\psi(y)$ and $\psi(x)$.
We construct a probability measure $\QQ=\QQ(\HH;x,y)$ on $\LC\times\LC$ as follows.
\begin{itemize}
\item
We sample a random point $u$ with density $\rho(u)$ given in~\eqref{eqn::dis}.
\item
Consider the Gaussian free field $\Gamma$ with the following boundary data: 
\[-\lambda\text{ on }(-\infty, 0), \quad \lambda\text{ on }(0, x),\quad (3+\nu)\lambda\text{ on }(x,y), \quad (\kappa-3)\lambda\text{ on }(y,u),\quad(\kappa-7)\lambda\text{ on }(u,+\infty). \]
\item
We define $\eta_1$ as follows: Before hitting $(y,+\infty)$, we define $\eta_1$ to be the flow line of the GFF $\Gamma$ starting from $0$ to $\infty$. Denote by $S$ the hitting time of $\eta_1$ at $(y,+\infty)$. After time $S$, we define $\eta_1$ to be an independent $\SLE_\kappa$ from $\eta_1(S)$ to $\infty$ in the unbounded connected component of $\HH\setminus\eta_1[0,S]$. 
\item
Define $\eta_2$ to be the flow line of $\Gamma-2\lambda$ from $x$ to $y$.
\end{itemize}
We denote by $\QQ$ the law of $(\eta_1,\eta_2)$. Under $\QQ$, random curves $(\eta_1,\eta_2)$ has the following properties.
\begin{itemize}
\item
By Theorem~\ref{thm::decom}, the marginal law of $\eta_1$ under $\QQ$ equals the law of $\hSLE(\nu)$ from $0$ to $\infty$ with force points $x$ and $y$. Moreover, under $\QQ$, almost surely, $\eta_1$ ends at $u$. 
\item
From the angle difference of the flow lines, almost surely, $\eta_2$ stays to the right of $\eta_1$. 
\begin{itemize}
\item
Given $u$, we have that $\eta_2$ hits $(y,+\infty)$ at $(y,u)$ and $\eta_2$ does not hit $(-\infty,0]\cup[u,+\infty)$ almost surely.
\item
Given $u$, the conditional law of $\eta_2$ equals $\SLE_\kappa(2;\nu,\kappa-6-\nu,-4)$ from $0$ to $\infty$ with marked points $0$ on its left and $x^+,y,u$ on its right, up to the hitting time of $(y,u)$. Denote by $T$ the hitting time of $\eta_2$ at $(y,u)$ and by $\sigma$ the last hitting time of $\eta_2$ at $(x,y)$. Denote by $\LC_T$ the connected component of $\HH\setminus\eta_2[0,T]$ whose boundary contains $y$. Then, after hitting $(y,u)$, the random curve $\eta_2$ evolves as $\SLE_\kappa(\nu)$ in $\LC_T$ from $\eta_2(T)$ to $y$ with marked point $\eta_2(\sigma)$.
\item
Given $u$ and $\eta_2$, the conditional law of $\eta_1$ equals $\SLE_\kappa(\kappa-4,-4)$ in the unbounded connected component of $\HH\setminus\eta_2$ from $0$ to $\infty$ with force points $\eta_2(T)$ and $u$.
\end{itemize}
\end{itemize}
In the following Lemma, we will derive the marginal law of $\eta_2$ under $\QQ$.
\begin{lemma}\label{lem::eta2rever}
Suppose $\eta$ is $\SLE_\kappa(\nu)$ from $x$ to $y$ with marked point $x^+$. Denote by $\LC_\eta$ the unbounded connected component of $\HH\setminus\eta$. The marginal law of $\eta_2$ under $\QQ$ equals the law of $\eta$ weighted by 
\[\one_{\{\eta\text{ does not hit }(-\infty,0]\}}\frac{B\left(\frac{\kappa-4}{\kappa},\frac{8-\kappa}{\kappa}\right)}{\mathcal{Z}}\times \phi'(0)^h,\]
where $\phi$ is any conformal map from $\LC_\eta$ onto $\HH$ such that $\phi(\infty)=\infty$ and $\phi'(\infty)=1$ and
\begin{equation}\label{eqn::defpartition}
\mathcal{Z}=\mathcal{Z}(x,y):=x^{\frac{\nu+2}{\kappa}}y^{\frac{\kappa-4-\nu}{\kappa}}(y-x)^{\frac{(\nu+2)(\kappa-4-\nu)}{2\kappa}}\int_y^{+\infty} s^{-\frac{4}{\kappa}}(s-y)^{\frac{2\nu+12-2\kappa}{\kappa}}\left(s-x\right)^{-\frac{2\nu+4}{\kappa}}ds.
\end{equation}
In particular, the time-reversal of ${\eta}_2$ equals the marginal law of $\eta_2$ under $\QQ(\HH;\psi(y),\phi(x))$.
\end{lemma}
To prove Lemma~\ref{lem::eta2rever}, we begin with an auxiliary lemma.
\begin{lemma}\label{lem::auxeta2}
Fix $x<y<u$. Suppose $\eta$ is $\SLE_\kappa(\nu)$ from $x$ to $y$ with marked point $x^+$. We denote by $W$ the driving function of $\eta$ and by $(g_t:t\ge 0)$ the corresponding conformal maps and by $(V_t(z):t\ge 0)$ the corresponding Loewner flow for every $z\in\R$. Denote by $T$ the hitting time of $\eta$ at $(y,u)$ and by $\tau$ the hitting time of $\eta$ at $(-\infty,0]\cup[u,+\infty)$. For $t<T\wedge\tau$, define
\begin{align*}
M_t=M_t(u):=&g'_t(0)^hg'_t(u)\times (W_t-V_t(0))^{\frac{2}{\kappa}}(V_t(x^+)-V_t(0))^{\frac{\nu}{\kappa}}(V_t(y)-g_t(0))^{\frac{\kappa-6-\nu}{\kappa}}(V_t(u)-V_t(0))^{-\frac{4}{\kappa}}\\
&\times(V_t(u)-W_t)^{-\frac{4}{\kappa}}(V_t(u)-V_t(x^+))^{-\frac{2\nu}{\kappa}}(V_t(u)-V_t(y))^{-\frac{2(\kappa-6-\nu)}{\kappa}}.
\end{align*}
Then, we have:
\begin{itemize}
\item
$\{M_{t\wedge T\wedge\tau}\}_{t\ge 0}$ is a uniformly integrable martingale for $\eta$. Moreover,
\begin{itemize}
\item
On the event $\{T<\tau\}$, we have
\begin{equation}\label{eqn::auxterminal1}
M_T=\lim_{t\to T}M_t=g'_T(0)^hg'_T(u)\left(W_T-g_T(0)\right)^{\frac{\kappa-4}{\kappa}}\left(g_T(u)-g_T(0)\right)^{-\frac{4}{\kappa}}\left(g_T(u)-W_T\right)^{-\frac{2(\kappa-4)}{\kappa}}.
\end{equation}
\item
On the event $\{\tau<T\}$, we have
\begin{equation}\label{eqn::auxterminal2}
M_\tau=\lim_{t\to \tau}M_t=0.
\end{equation}
\end{itemize}
\item
Suppose $\eta^u$ is $\SLE_\kappa(2;\rho,\kappa-6-\rho,-4)$ from $0$ to $\infty$ with marked points $0$ on its left and $x^+,y,u$ on its right. Then, the law of $\eta$ weighted by $\left\{\frac{M_{t\wedge T\wedge\tau}}{M_0}\right\}_{t\ge 0}$ equals the law of $\eta^u$, up to the hitting time at $(y,u)$.
\end{itemize}
\end{lemma}
\begin{proof}
By definition of $\eta$ and~\cite[Theorem 3]{SchrammWilsonSLECoordinatechanges}, before time $T$, the law of $\eta$ equals the law of $\SLE_\kappa(\nu,\kappa-6-\nu)$ from $x$ to $\infty$ with marked points $x^+$ and $y$. Thus,~\cite[Theorem 6]{SchrammWilsonSLECoordinatechanges} implies that $\{M_{t\wedge T\wedge\tau}\}_{_t\ge 0}$ is a local martingale for $\eta$ and the law of $\eta$ weighted by $\left\{\frac{M_{t\wedge T\wedge\tau}}{M_0}\right\}_{t\ge 0}$ equals the law of $\eta^u$ before hitting $(-\infty,0)\cup(y,+\infty)$.

Firstly, we prove that $\{M_{t\wedge T\wedge\tau}\}_{t\ge 0}$ is uniformly integrable. Define $\tau_n$ to be the hitting time of curves at the union of $\frac{1}{n}$-neighborhood of $(-\infty,0)\cup(u,+\infty)$ and $ \partial B(0,n)$. Then, we have that
$\{M_{t\wedge T\wedge\tau_n}\}$ is uniformly bounded. Then, $\{M_{t\wedge \tau_n}\}_{t\ge 0}$ is a uniformly integrable martingale. Denote by $P_n$ the law of $\eta$ weighted by $\left\{\frac{M_{t\wedge T\wedge\tau_n}}{M_0}\right\}_{t\ge 0}$. Then $P_n$ is the same as the law of $\eta^u$ before $T\wedge\tau_n$. Since $\{P_n\}_{n\ge 1}$ is compatible, there exists a probability measure $P$ on curves, such that under $P$, the law of the random curve equals $\eta^u$ before hitting $(-\infty,0]\cup[u,+\infty)$. Since $\eta^u$ does not hit $(-\infty,0]\cup[u,+\infty)$ almost surely, we have that $P$ is the same as the law of $\eta^u$ up to $T$. This implies that $\left\{\frac{M_{t\wedge T\wedge\tau}}{M_0}\right\}_{t\ge 0}$ is uniformly integrable and the law of $\eta$ weighted by $\left\{\frac{M_{t\wedge T\wedge\tau}}{M_0}\right\}_{t\ge 0}$ equals the law of $\eta^u$, up to the hitting time of $(y,u)$.

Secondly, we derive~\eqref{eqn::auxterminal1} and~\eqref{eqn::auxterminal2}. 
\begin{itemize}
\item
On the event $\{T<\tau\}$, we obtain~\eqref{eqn::auxterminal1} from the continuity of $V_t$ and $g_t$.
\item
On the event $\{\tau<T\}$, we have
\[\E\left[\frac{M_\tau}{M_0}\one_{\{\tau<T\}}\right]=\PP[\eta^u\text{ hits }(-\infty,0]\cup[u,+\infty)]=0.\]
This implies~\eqref{eqn::auxterminal2}.
\end{itemize}
Thus, we finish the proof.
\end{proof}
\begin{proof}[Proof of Lemma~\ref{lem::eta2rever}]
We use the same notations as in Lemma~\ref{lem::auxeta2}. Recall that given $u$, the conditional law of $\eta_2$ equals $\eta^u$ before hitting $(y,u)$. After hitting $(y,u)$, it evolves as $\SLE_\kappa(\nu)$. Suppose $G$ is a bounded and measurable function on the curve space. 
By Lemma~\ref{lem::auxeta2}, we have
\begin{align*}
\QQ[G(\eta_2)]&=\E\left[\int^{+\infty}_y du\frac{M_{T}(u)}{M_0(u)}\one_{\{\eta\text{ does not hit }(-\infty,0]\cup[u,+\infty)\}}G(\eta)\rho(u)\right]\\
&=\frac{1}{\mathcal{Z}}\E\left[G(\eta)\one_{\{\eta\text{ does not hit }(-\infty,0]\}}g'_T(0)^h\times\left(W_T-g_T(0)\right)^{\frac{\kappa-4}{\kappa}}\int^{+\infty}_{W_T}ds\left(s-g_T(0)\right)^{-\frac{4}{\kappa}}\left(s-W_T\right)^{-\frac{2(\kappa-4)}{\kappa}}\right]\\
&=\frac{B\left(\frac{\kappa-4}{\kappa},\frac{8-\kappa}{\kappa}\right)}{\mathcal{Z}}\E\left[G(\eta)\one_{\{\eta\text{ does not hit }(-\infty,0]\}}g'_T(0)^h\right].
\end{align*}
We still need to prove the reversibility of $\eta_2$. Note that by~\cite[Theorem 1.2]{MillerSheffieldIG3}, the law of $\psi(\eta)$ equals $\SLE_\kappa(\nu)$ from $\psi(y)$ to $\psi(x)$.
We define $g_\psi:=g'_T(0)\times\psi\circ g_T\circ \psi^{-1}$. Then, $g_\psi$ is the conformal map from the unbounded connected component of $\HH\setminus\psi(\eta_2)$ onto $\HH$ such that $g_\psi(\infty)=\infty$ and $g'_\psi(\infty)=1$. Note that $g'_\psi(0)=g'_T(0)$. Thus, we have
\begin{equation}\label{eqn::revereta2}
\QQ(\HH;x,y)[G(\psi({\eta}_2))]=\frac{\mathcal{Z}(\psi(y),\psi(x))}{\mathcal{Z}(x,y)}\QQ(\HH;\psi(y),\psi(x))[G(\eta_2)].
\end{equation}
By taking $G\equiv 1$, we have $\mathcal{Z}(\psi(y),\psi(x))=\mathcal{Z}(x,y)$. Combining with~\eqref{eqn::revereta2}, we complete the proof.
\end{proof}
In the next lemma, we will derive the conditional law of $\eta_1$ given $\eta_2$.
\begin{lemma}\label{lem::condieta1}
Under $\QQ$, the conditional law of $\eta_1$ given $\eta_2$ equals $\SLE_\kappa$ from $0$ to $\infty$ in the unbounded connected component of $\HH\setminus\eta_2$.
\end{lemma}
\begin{proof}
Suppose $G_1$ and $G_2$ are two bounded and continuous functions on curve space. We use the same notations as in Lemma~\ref{lem::auxeta2}. Recall the definition of $\rho(u)$ in~\eqref{eqn::dis} and the definition of $\mathcal{Z}$ in~\eqref{eqn::defpartition}. Recall that we denote by $S$ the hitting time of $\eta_1$ at $(y,+\infty)$ and by $T$ the hitting time of $\eta_2$ at $(y,+\infty)$.  Define $v=g_T(u)$, by construction of $\QQ$ and Lemma~\ref{lem::auxeta2}, we have
\begin{align}\label{eqn::auxcond}
&\QQ[G_1(g_T(\eta_1[0,S]))G_2(\eta_2)]\notag\\
=&\QQ[\QQ[\QQ[G_1(g_T(\eta_1[0,S]))\cond \eta_2,u]G_2(\eta_2)\cond u]]\notag\\
=&\E\left[\QQ[G_1(g_T(\eta_1[0,S]))\cond \eta,u]G_2(\eta)\one_{\{\eta\cap(-\infty,0]=\emptyset\}}\int^{+\infty}_{W_T}du\rho(u)\frac{M_{T\wedge\tau}(u)}{M_0(u)}\right]\notag\\
=&\frac{1}{\mathcal{Z}}\E\left[G_2(\eta)\one_{\{\eta\cap(-\infty,0]=\emptyset\}}g'_T(0)^h\underbrace{\left(W_T-g_T(0)\right)^{\frac{\kappa-4}{\kappa}}\int^{+\infty}_{W_T}ds \QQ[G_1(g_T(\eta_1[0,S]))\cond \eta,v]\left(v-g_T(0)\right)^{-\frac{4}{\kappa}}\left(v-W_T\right)^{-\frac{2(\kappa-4)}{\kappa}}}_R\right].
\end{align}
Note that by construction, given $\eta$ and $v$, the conditional law of $g_T(\eta_1)$ equals $\tilde\eta^v\sim\SLE_\kappa(\kappa-4,-4)$ from $g_T(0)$ to $\infty$ with marked points $W_T$ and $v$, up to the hitting time at $(W_T,+\infty)$. 
Suppose $\tilde\eta$ is $\SLE_\kappa$ from $g_T(0)$ to $\infty$. Denote by $\tilde W$ the driving function of $\tilde\eta$ and by $(\tilde g_s:s\ge 0)$ the corresponding conformal maps. Define $\tilde S$ to be the hitting time of $\tilde\eta$ at $(W_T,+\infty)$. For $s<\tilde S$, we define
\[\tilde N_s=\tilde N_s(v):=\left(\tilde g_s(W_T)-\tilde W_s\right)^{\frac{\kappa-4}{\kappa}}\left(v-\tilde W_s\right)^{-\frac{4}{\kappa}}\left(v-\tilde g_s(W_T)\right)^{-\frac{2(\kappa-4)}{\kappa}}\tilde g'_s(v).\]
By~\cite[Theorem 6]{SchrammWilsonSLECoordinatechanges}, we have that $\{\tilde N_{s\wedge \tilde S}\}_{s\ge 0}$ is a local martingale for $\tilde \eta$ and the law of $\tilde \eta$ weighted by $\left\{\frac{\tilde N_{s\wedge \tilde S}}{\tilde N_0}\right\}_{s\ge 0}$ equals the law of $\tilde\eta^v$ before hitting $(W_T,+\infty)$.
Define $\tilde S_n$  to be the hitting time of $\tilde\eta$ at the union of $\frac{1}{n}$-neighborhood of $(W_T,+\infty)$ and $\partial B(0,n)$. Note that $\tilde N_s$ is uniformly bounded for $s\le \tilde S_n$. Then, we have that $\{\tilde N_{s\wedge\tilde S_n}\}_{s\ge 0}$ is uniformly integrable. If we define $S_n$ to be the hitting time of $g_T(\eta_1)$ at the union of $\frac{1}{n}$-neighborhood of $(W_T,+\infty)$ and $\partial B(0,n)$, then we have
\begin{align}
R&=\lim_{n\to+\infty}\left(W_T-g_T(0)\right)^{\frac{\kappa-4}{\kappa}}\int^{+\infty}_{W_T}dv \QQ[G_1(g_T(\eta_1[0,S_n]))\cond \eta,v]\left(v-g_T(0)\right)^{-\frac{4}{\kappa}}\left(v-W_T\right)^{-\frac{2(\kappa-4)}{\kappa}}\tag{due to the continuity of $\eta_1$}\\
&=\lim_{n\to+\infty} \E\left[G_1(\tilde\eta[0,\tilde S_n])\int^{+\infty}_{W_T}dv \tilde N_{\tilde S_n}(v)\right]\tag{due to the uniform integrability of $\{\tilde N_{s\wedge\tilde S_n}\}_{s\ge 0}$}\\
&=B\left(\frac{\kappa-4}{\kappa},\frac{8-\kappa}{\kappa}\right)\lim_{n\to+\infty} \E\left[G_1(\tilde\eta[0,\tilde S_n])\right]\notag\\
&=B\left(\frac{\kappa-4}{\kappa},\frac{8-\kappa}{\kappa}\right)\E\left[G_1(\tilde\eta[0,\tilde S])\right].\tag{due to the continuity of $\tilde\eta$}
\end{align}
Plugging into~\eqref{eqn::auxcond}, combining with Lemma~\ref{lem::eta2rever}, we have
\begin{align}\label{eqn::auxfinal}
\QQ[G_1(g_T(\eta_1[0,S]))G_2(\eta_2)]&=\frac{B\left(\frac{\kappa-4}{\kappa},\frac{8-\kappa}{\kappa}\right)}{\mathcal{Z}}\E\left[G_2(\eta)\one_{\{\eta\cap(-\infty,0]=\emptyset\}}g'_T(0)^h\E[G_1(\tilde\eta[0,\tilde S])]\right]\notag\\
&=\QQ\left[G_2(\eta_2)\E[G_1(\tilde\eta[0,\tilde S])]\right].
\end{align}
By definition, under $\QQ$, given $\eta_1[0,S]$ and $\eta_2$, the random curve $\eta_1[S,+\infty)$ evolves as $\SLE_\kappa$ in the unbounded connected component of $\HH\setminus\eta_1[0,S]$ from $\eta_1(S)$ to $\infty$. Combining with~\eqref{eqn::auxfinal}, we complete the proof.
\end{proof}
\begin{proof}[Proof of Theorem~\ref{thm::reversibility}]
By Lemma~\ref{lem::condieta1} and~\cite[Theorem 1.2]{MillerSheffieldIG3}, we have that under $\QQ$, given $\eta_2$, the conditional law of $\psi({\eta}_1)$ equals $\SLE_\kappa$ in the unbounded connected component of $\HH\setminus\psi({\eta}_2)$ from $0$ to $\infty$. By Lemma~\ref{lem::eta2rever}, we have that the law of $\psi({\eta}_2)$ equals the law of $\eta_2$ under $\QQ(\HH;\psi(y),\psi(x))$. Combining with Lemma~\ref{lem::eta2rever} and Lemma~\ref{lem::condieta1}, we have that the law of $\psi({\eta}_1)$ equals $\hSLE_\kappa(\nu)$ from $0$ to $\infty$ with force points $\psi(y)$ and $\psi(x)$. This completes the proof.
\end{proof}
\end{document}